\documentclass{article}
\usepackage{amsmath,amssymb,amsthm}

\newtheorem{thm}{Theorem}[section]
\newtheorem{cor}[thm]{Corollary}
\newtheorem{prop}[thm]{Proposition}
\newtheorem{lemma}[thm]{Lemma}
\theoremstyle{definition}
\newtheorem{dfn}[thm]{Definition}
\newtheorem{ex}[thm]{Example}
\newtheorem*{remark}{Remark}
\newtheorem*{remarks}{Remarks}
\newtheorem*{note}{Note}

\newtheorem*{ack}{Acknowledgements}

\def\eq#1{{\rm(\ref{#1})}}

\def\R{{\mathbb R}}

\def\GL{\mathbin{\rm GL}}

\def\SU{\mathop{\rm SU}}
\def\SO{\mathop{\rm SO}}

\def\G2{\mathop\textrm{G}_2}

\def\d{{\rm d}}
\def\w{\wedge}

\DeclareMathOperator{\id}{id}

\def\RE{{\mathbb R}}

\def\bfn{\mathbf{n}}

\def\bfv{\mathbf{v}}

\def\HH{\mathcal{H}}

\def\p{\partial}
\def\we{\wedge}
\def\eps{\varepsilon}
\let\phi=\varphi

\begin{document}

\title{Deformations of Compact Coassociative 4-folds with Boundary}
\author{
\begin{minipage}{2.25in}\begin{center}
\textsc{Alexei Kovalev}\\ DPMMS\\ University of Cambridge
\end{center}\end{minipage}
\begin{minipage}{2.25in}\begin{center}
\textsc{Jason D. Lotay}\\ University College\\ Oxford
\end{center}\end{minipage}\\
\rule{1ex}{0ex}
}
\maketitle

\section{Introduction}

Coassociative 4-folds are a particular class of 4-dimensional submanifolds
which may be defined in a 7-dimensional manifold $M$ endowed with a
`$\G2$ form' $\phi$. The latter is a differential 3-form which is invariant
at each point under the action of the exceptional Lie group $\G2$. This
3-form induces a $\G2$-\emph{structure} on~$M$ and, consequently, a
Riemannian metric and orientation. If the form $\phi$ is coclosed then every
coassociative 4-fold in~$M$ is calibrated and hence minimal. See \S\ref{s2}
for precise definitions and a summary of the relevant theory. Coassociative
submanifolds were introduced by Harvey and Lawson~\cite{HarLaw} who also
gave $\SU(2)$-invariant examples of these submanifolds in Euclidean
$\RE^7$. Examples of compact coassociative submanifolds of compact
7-manifolds with holonomy $\G2$ were given by Joyce~\cite{Joyce1} and later
by the first author~\cite{kov}.

McLean~\cite{McLean} showed that, when the $\G2$-structure 3-form
is closed, the deformations of a compact coassociative 4-fold without
boundary are \emph{unobstructed} and the moduli space of local deformations is
smooth with dimension determined by the \emph{topology} of this submanifold.

There is some analogy between coassociative 4-folds and special
Lagrangian submanifolds of Calabi--Yau manifolds. Both classes may be
defined by the vanishing of appropriate differential forms on the ambient
manifold. Compact closed special Lagrangian submanifolds also have an
unobstructed deformation theory and finite-dimensional moduli space with
topologically determined dimension, a result due again to
Mclean~\cite{McLean}. Butscher~\cite{Butscher} extended this result
to compact minimal Lagrangian submanifolds with boundary and showed that
these have a finite-dimensional moduli space of deformations if the
boundary lies in an appropriately chosen symplectic submanifold
which he called a {\em scaffold}.

On a 7-manifold endowed with a $\G2$-structure, there is also a distinguished class 
of 3-dimensional submanifolds known as associative 3-folds.  Recently, compact associative 
3-folds with boundary in a coassociative 4-fold were studied in \cite{Gayet}.  Here the deformation theory 
is obstructed, but the expected dimension of the moduli space is given in terms of the boundary of 
the associative 3-fold.

In this article, we study the deformations of compact coassociative 4-folds
$N$ with boundary in an particular fixed 6-dimensional submanifold
$S\subset M$ which, by analogy with~\cite{Butscher}, we also call a
scaffold (Definition~\ref{s3dfn2}). The condition on $S$ is that it has a
Hermitian symplectic structure compatible with the $\SU(3)$-structure it
inherits from $M$. We also require that the normal vectors to~$S$ at $\p N$
are tangent to~$N$. The culmination of the research presented here is the
following two theorems.

\begin{thm}\label{thm1}
Suppose that $M$ is a 7-manifold with a $\text{\emph{G}}_2$-structure given
by a closed 3-form. The moduli space of compact coassociative local deformations
of~$N$ in $M$ with boundary $\p N$ in a scaffold $S$ is a
finite-dimensional smooth manifold of dimension not greater than~$b^1(\p N)$.
\end{thm}

\begin{thm}\label{thm2}
Let $\phi(t)$ be a smooth 1-parameter family of closed 3-forms defining
$\text{\emph{G}}_2$-structures on~$M$. Suppose that a compact submanifold
$N\subset M$ with boundary is coassociative with respect to the
$\text{\emph{G}}_2$-structure of $\phi(0)$ and the boundary $\p N$
is contained in a scaffold $S$.

If $\phi(t)|_N$ and the normal part of $\phi(t)|_{N}$ on $\p N$ are exact 
for all $t$ then $N$ can be extended to a smooth
family $N(t)$ for small $|t|$, with $N(0)=N$, such that $N(t)$ is
coassociative with respect to $\phi(t)$ and the boundary of $N(t)$ is
in~$S$.
\end{thm}

The principal ingredient in the proof of Theorems \ref{thm1} and~\ref{thm2}
is the construction of an appropriate boundary value problem with Fredholm
properties. For geometric reasons, the boundary value problem for
coassociative 4-folds with boundary \emph{cannot} be of a standard
Dirichlet or Neumann type (see remark on page~\pageref{eight}). Our study
of coassociative deformations leads to a boundary value problem of second
order and altogether quite different from that for the minimal Lagrangians
in~\cite{Butscher}, which is Neumann first order. 

We set-up the infinitesimal deformation
problem for coassociative submanifolds with boundary in a scaffold
in \S\ref{s4}, where we also study the Fredholm properties and give
a version of the Tubular Neighbourhood Theorem which is adapted to our needs. Then, in
$\S$\ref{s4}, we define a `deformation map' and apply the
Implicit Function Theorem to it in order to prove Theorems \ref{thm1} and~\ref{thm2}.  We
also briefly discuss some applications of the deformation theory in
$\S$\ref{examples}.

\begin{note}
Submanifolds are taken to be embedded, for convenience, since the
results hold for immersed submanifolds by simple modification of the
arguments given.
Smooth functions (and, more generally, sections of vector
bundles) on $N$ are understood as `smooth up to the boundary', i.e.\ near
$\p N$ these are obtainable as restrictions to $\p N\times[0,\eps_N)$
of smooth functions (sections) defined on $\p N\times(-\eps_N,\eps_N)$, via
the diffeomorphism~\eqref{collarn}.
\end{note}

\section{Coassociative 4-folds}\label{s2}

The key to defining coassociative
4-folds lies with the introduction of a distinguished 3-form on the
Euclidean space $\R^7$.

\begin{dfn}\label{s2dfn1} Let $(x_1,\ldots,x_7)$ be coordinates on $\R^7$ and write
$\d{\bf x}_{ij\ldots k}$ for the form $\d x_i\w \d x_j\w\ldots\w \d x_k$.
Define a 3-form $\varphi_0\in\Lambda^3(\RE^7)^*$ by:
\begin{equation}\label{s2eq1}
\varphi_0 = \d{\bf x}_{123}+\d{\bf x}_{145}+\d{\bf x}_{167}+\d{\bf
x}_{246}- \d{\bf x}_{257}-\d{\bf x}_{347}-\d{\bf x}_{356}.
\end{equation}
\end{dfn}
The Hodge dual of $\varphi_0$ is a 4-form given by:
\begin{equation*}\label{s2eq2}
\ast\varphi_0 = \d{\bf x}_{4567}+\d{\bf x}_{2367}+\d{\bf
x}_{2345}+\d{\bf x}_{1357}-\d{\bf x}_{1346}-\d{\bf x}_{1256}-\d{\bf
x}_{1247}.
\end{equation*}
The usual basis of $\R^7$ may be identified with the standard basis of
the cross-product algebra of pure imaginary octonions. Then
\begin{equation}\label{s2eq0}
\varphi_0(\mathbf{x},\mathbf{y},\mathbf{z})=
g_0(\mathbf{x}\times\mathbf{y},\mathbf{z})
\qquad\text{for any }\mathbf{x},\mathbf{y},\mathbf{z}\in\R^7,
\end{equation}
where $g_0$ denotes the Euclidean metric. 

The subgroup of $\GL(7,\R)$ preserving the cross-product is the Lie group
$\G2$, which is also a subgroup of $\SO(7)$, so $\G2$ is the stabilizer of
$\phi_0$ in the action of $\GL(7,\R)$. In light of this property, 
$\phi_0$ is sometimes called a `$\G2$ 3-form' on $\RE^7$; our choice of
expression~\eq{s2eq1} for $\varphi_0$ follows that of
\cite[Definition 11.1.1]{Joyce2}.
\begin{remark}
We note, for later use, that the stabilizer of a non-zero vector
$\mathbf{e}\in\RE^7$ in the action of $\G2$ is a maximal subgroup of~$\G2$
isomorphic to $\SU(3)$. Thus $\SU(3)$ is the stabilizer of a pair $(\omega,\Upsilon)$, for 
a 2-form $\omega=(\mathbf{e}\lrcorner\phi_0)|_{\mathbf{e}^\bot}$ and a 3-form
$\Upsilon=\phi_0|_{\mathbf{e}^\bot}$, in the action of $\GL(6,\RE)$ on
the orthogonal complement $\mathbf{e}^\bot\cong\RE^6$ (cf.~\cite{hitchin}).
\end{remark}
\begin{dfn}\label{s2dfn2} A 4-dimensional submanifold $P$
of\/ $\R^7$ is {\em coassociative} if and only if $\varphi_0|_P\equiv 0$.
\end{dfn}
The condition $\varphi_0|_P\equiv 0$ forces $*\varphi_0$ to be a
non-vanishing 4-form on $P$. Thus $*\varphi_0|_P$ induces a canonical
{\em orientation} on~$P$.

Definition~\ref{s2dfn2} is equivalent to the 
definition used in \emph{calibrated geometry} \cite{Harvey,Joyce2}. That is,
a coassociative 4-fold $P$ is a submanifold calibrated by $*\varphi_0$,
which means that $*\varphi_0|_P$ is the volume form for the Riemannian metric
induced by $g_0$ on~$P$ with the canonical orientation
\cite[Proposition 12.1.4]{Joyce2}. Every coassociative submanifold
of~$\RE^7$ is a {\em minimal} submanifold \cite[Theorem II.4.2]{HarLaw}
and, moreover, volume minimizing \cite[Theorem 7.5]{Harvey}.

So that we may describe coassociative submanifolds of more general
7-manifolds, we make two definitions following \cite[p.~7]{Bryant4}.

\begin{dfn}\label{s2dfn3}
Let $M$ be an oriented 7-manifold. A 3-form $\varphi$ on $M$ is
\emph{positive} if $\varphi(x)=\iota_x^*(\varphi_0)$ for all $x\in M$
for some orientation preserving linear isomorphism
$\iota_x:T_xM\rightarrow\R^7$, where $\phi_0$ is given in~\eqref{s2eq1}.
Denote the subbundle of positive \mbox{3-forms} on~$M$ by
$\Lambda^3_+T^*M \subset \Lambda^3T^*M$ and the fibre of this subbundle over
$x\in M$ by $\Lambda^3_+T^*_xM$. We write $\Omega^3_+(M)$ for the space of all
(smooth) positive 3-forms on~$M$.
\end{dfn}
For each $x\in M$, $\Lambda^3_+T^*_xM$ is the image of the open
$\GL_+(7,\R)$-orbit of $\phi_0$ in $\Lambda^3{(\RE^7)}^*$
under~$\iota_x^*$ given in the above definition. It follows that
$\Lambda^3_+T^*M$ is an open subbundle of $\Lambda^3T^*M$.

Since a positive 3-form $\phi$ is identified at each point in $M$ with the
3-form $\phi_0$ stabilized by $\G2$, it determines a $\G2$-structure on~$M$.
We shall sometimes simply say that $\phi\in\Omega^3_+(M)$ {\em is} a
$\G2$-structure on the oriented 7-manifold $M$.

Furthermore, as $\G2\subset \SO(7)$, we can uniquely associate to each
$\phi\in\Omega^3_+(M)$ a Riemannian metric $g=g(\phi)$ and the Hodge dual
4-form $*\varphi$ relative to the Hodge star of $g$. The triple
$(\varphi,*\varphi,g)$ corresponds to $(\varphi_0,*\varphi_0,g_0)$
at each point. 

\begin{dfn}[\mbox{\cite[pp.~228, 264]{Joyce2}}]
Let $M$ be an oriented 7-manifold. We call a $\G2$-structure
$\phi\in\Omega^3_+(M)$ \emph{torsion-free} if $\phi$ is closed
and coclosed with respect to the induced metric $g(\phi)$.

An \emph{almost} $\G2$-\emph{manifold} $(M,\varphi)$ is a 7-manifold
endowed with a $\G2$-structure $\phi$ such that $d\phi=0$.
If a $\G2$-structure $\phi$ is torsion-free then
$(M,\varphi)$ is called a $\G2$-\emph{manifold}.
\end{dfn}

\begin{remark}
By \cite[Lemma 11.5]{Salamon}, the holonomy of a metric $g$ on~$M$ is
contained in $\G2$ if and only if $g=g(\phi)$ for some torsion-free
$\G2$-structure $\phi$ on $M$. Manifolds with a closed $\G2$ 3-form are
important in the constructions of examples of compact irreducible
$\G2$-manifolds in \cite{Joyce1} and \cite{kov}.
\end{remark}

We are now able to give a more general version of Definition~\ref{s2dfn2}.

\begin{dfn}\label{s2dfn6}
Let $M$ be an oriented 7-manifold and $\phi\in\Omega^3_+(M)$ a $\G2$-structure
on~$M$. A 4-dimensional submanifold $P$ of\/ $M$ is {\em coassociative} if and
only if\/ $\varphi|_P\equiv 0$.
\end{dfn}

The deformation theory of compact coassociative submanifolds was studied by
McLean~\cite{McLean}. His results, summarized in Theorem~\ref{mclean} below,
were stated for $\G2$-manifolds but it was later observed in \cite{Goldstein}
that the proof does not use the coclosed condition on the $\G2$ 3-form~$\phi$.

By way of preparation, we note a standard consequence of the proof of the
Tubular Neighbourhood Theorem in \cite[Chapter IV, Theorem~9]{Lang}.
\begin{prop}\label{s3thm1}
Let $P$ be a closed submanifold
of a Riemannian manifold $M$.  There exist an open subset $V_P$ of
the normal bundle $\nu_{M}(P)$ of $P$ in $M$, containing the zero
section, and a tubular neighbourhood $T_P$ of $P$ in $M$, such that the
exponential map $\exp_{M}|_{V_P}:V_P\rightarrow T_P$ is a
diffeomorphism onto~$T_P$.
\end{prop}
The local deformations of~$P$ are understood as
submanifolds of the form $\exp_{\bf{v}}(P)$, where $\bfv$ is a
$C^1$-section of the normal bundle $\nu_M(P)$ and $\bfv$ is assumed
sufficiently small in the sup-norm. We shall call sections of $\nu_M(P)$ the
{\em normal vector fields} at~$P$.

Now suppose that $(M,\phi)$ is an almost $\G2$-manifold and a submanifold
$P\subset M$ is coassociative. Then the local deformations of~$P$ may
equivalently be given by self-dual 2-forms on~$P$ using an isometry of vector
bundles (cf.~\cite[Proposition 4.2]{McLean})
\begin{equation}\label{s2prop2}
\jmath_P:\mathbf{v}\in\nu_M(P)\to
(\mathbf{v}\,\lrcorner\,\varphi)|_P\in\Lambda^2_+T^*P;
\end{equation}
where $\Lambda^2_+T^*P$ denotes the bundle of self-dual 2-forms. The map
\begin{equation}\label{mapF}
F:\alpha\in\Omega^2_+(P)\to  \exp_{\bfv}^*(\varphi)\in\Omega^3(P),
\qquad  \bfv=\jmath_P^{-1}(\alpha),
\end{equation}
is defined for `small' $\alpha$, and $F(\alpha)=0$ precisely if
$\exp_{\bf{v}}(P)$ is a coassociative deformation.
\begin{thm}[\mbox{cf.~\cite[Theorem 4.5]{McLean}, \cite[Theorem 2.5]{joyce-salur}}]\label{mclean}
Let $(M,\varphi)$ be an almost $\text{\emph{G}}_2$-manifold
and $P\subset M$ a coassociative submanifold (not necessarily closed). 
\begin{itemize}
\item[\emph{(a)}] Then for each $\alpha\in\Omega^2_+(P)$, one has
$dF|_0(\alpha)=d\alpha$ and the 3-form 
$F(\alpha)$ (if defined) is exact.
\item[\emph{(b)}] If, in addition, $P$ is compact and without boundary then every closed
self-dual 2-form $\alpha$ on~$P$ arises as $\alpha=\jmath_P(\bfv)$, for some
normal vector field~$\bfv$ tangent to a smooth 1-parameter family of
coassociative submanifolds containing~$P$.  Thus, in this case, the space of
local coassociative deformations of~$P$ is a smooth manifold parameterized by
the space $\HH^2_+(P)$ of closed self-dual 2-forms on~$P$.
\end{itemize}
\end{thm}
\begin{remark}
Self-dual 2-forms on a compact manifold without boundary are closed precisely
if they are harmonic. By Hodge theory, the dimension of $\HH^2_+(P)$ is therefore 
equal to the dimension $b^2_+(P)$ of a maximal positive subspace for the
intersection form on~$P$. It is thus a topological quantity.
\end{remark}
Finally, there is a useful extension of Theorem~\ref{mclean} in the situation
where the $\G2$-structure is allowed to vary. This result is stated
in~\cite[Theorem 12.3.6]{Joyce2} and can be proved using the techniques
of~\cite[\S 4]{McLean}.
\begin{thm}\label{extn-ca}
Let $\phi(t)\in\Omega^3_+(M)$, $t\in\RE$, be a smooth path of
closed $\text{\emph{G}}_2$-structure forms on~$M$. Suppose that $P$ is a compact
submanifold of~$M$ without boundary such that $\phi(0)|_P=0$ and the form
$\phi(t)|_P$ is exact for each~$t$. 

Then there is an $\eps>0$ and, for each
$|t|<\eps$, a section $\bfv(t)$ of $\nu_M(P)$ smoothly depending on~$t$,
such that $\bfv(0)=0$ and $\phi(t)$ vanishes on the submanifold
$P(t)=\exp_{\bfv(t)}(P)$.
Thus $P(t)$ is a coassociative 4-fold in $\big(M,\phi(t)\big)$.
Here the normal bundle $\nu_M(P)$ and the exponential map are understood with
respect to the metric induced by $\phi(0)$.
\end{thm}
\noindent Roughly speaking, this result says that compact coassociative 4-folds are `stable' under 
small variations of the ambient closed $\G2$-structure.

\section{The Infinitesimal Deformation Problem}\label{s3}

Throughout this section, $(M,\phi)$ is an almost $\G2$-manifold, $g=g(\phi)$
is the metric induced by~$\phi$ and $N$ is a compact coassociative submanifold
of $M$.

\subsection{{\boldmath $\SU(3)$}-structures on 6-dimensional submanifolds}
\label{6submfd}

In order to explain and motivate our choice of the boundary condition in the
next subsection, we recall some results on $\SU(3)$-structures and Calabi--Yau
geometry and their relation to $\G2$~geometry.

To begin, suppose that $S$ is an orientable 6-dimensional submanifold
of~$M$. Then the normal bundle of $S$ is trivial and, by the Tubular
Neighbourhood Theorem (Proposition \ref{s3thm1}), there exists a
neighbourhood $T_S$ of $S$ which is 
diffeomorphic to $S\times \{-\eps_S<s<\eps_S\}$, for some $\eps_S>0$, so
that $S=\{s=0\}$ and $\bfn_S=\frac\p{\p s}$ is a unit vector field on $T_S$,
with $\bfn_S|_S$ orthogonal to $S$. We shall sometimes call  $s$ the normal
coordinate near $S$. We can write
\begin{equation}\label{localphi}
\varphi|_{T_S}=\omega_s\w\d s+\Upsilon_s
\end{equation}
for some 1-parameter family of 2-forms and 3-forms $\omega_s$ and
$\Upsilon_s$ on $S$.

It is not difficult to see, from the remark following Definition \ref{s2dfn1}, 
that the forms $\omega_0=(\bfn_S\lrcorner\phi)|_S$ and
$\Upsilon_0=\phi|_S$ together induce an $\SU(3)$-structure (in general,
with torsion) on $S$. In particular, $S$ is oriented by
$$\frac{\omega_0^3}{3!}=\frac{1}{4}\Upsilon_0\we *_6\Upsilon_0$$ and has an induced almost 
complex structure $I$, compatible with the orientation, which may be given
by $I(\mathbf{u})=\bfn_S\times\mathbf{u}$ for all 
$\mathbf{u}\in T_xS\subset T_xM$.  Here we denoted by $*_6$ the Hodge star on
the 6-manifold $S$ with respect to the induced metric from $M$, and we used the
cross-product on $T_xM$ given by a $\G2$-invariant identification with the
standard $\G2$ 3-form on~$\RE^7$ as in Definition~\ref{s2dfn3}.  

The metric
induced on~$S$ is Hermitian with respect to~$I$ and its fundamental
$(1,1)$-form is $\omega$.  The non-vanishing complex 3-form
$\Omega_0=*_6\Upsilon_0-i\Upsilon_0$ has type $(3,0)$ relative to~$I$. 
We also note that, as $\d s$ is a unit 1-form, the point-wise model~\eqref{s2eq1}
for $\phi$ yields the relation
\begin{equation}\label{ric-flat}
\omega_0^3=\frac{3i}{4}\Omega_0\w\bar{\Omega}_0
\end{equation}
at each point in~$S$. 

Denote by $d_S$ the exterior derivative on~$S$.
\begin{lemma}\label{SLCA}
Let $S$ be an oriented 6-dimensional submanifold of an almost
$\text{\emph{G}}_2$-manifold $(M,\phi)$ and suppose that $d_S\,\omega_0=0$,
where $\omega_0\in\Omega^2(S)$ is defined in~\eqref{localphi}.
Let $N\subset M$ be a coassociative submanifold intersecting
$S$ in a 3-dimensional submanifold $L$ and such that $\bfn_S|_L$ is tangent
to~$N$. Then $\omega_0|_L=0$ and $\Upsilon_0|_L=0$.
\end{lemma}
\begin{pf}
The last assertion is clear as $\Upsilon_0|_L=\phi|_L$, $L\subset N$
and $\phi|_N=0$. As
$$
0=d\phi|_{T_S}=(d_S\omega_s-\frac\p{\p s}\Upsilon_s)\we \d s+d_S\Upsilon_s
$$
we find that $\frac\p{\p s}\Upsilon_s|_{s=0}=0$. The submanifolds $S$ and
$N$ intersect transversely, so $\d s$ never vanishes on a neighbourhood
of $L$ in $N$. Restricting~\eqref{localphi} to this neighbourhood we
obtain $\omega_0|_L=0$ as $\bfn_S$ is a normal vector field at~$S$.
\end{pf}
The property $\omega_0|_L=0$ means that $L$ is a Lagrangian
submanifold of a symplectic manifold $(S,\omega_0)$.  To explain the role
of the additional condition $\Upsilon_0|_L=\phi|_L=0$, we begin with the
following.
\begin{dfn}\label{dfn3.2}
An orientable 6-dimensional submanifold $S$ is a
\emph{symplectic submanifold} of the almost $\G2$-manifold $(M,\phi)$ if
$d_S\omega_0=0$, where $\omega_0=(\bfn_S\lrcorner\phi)|_S$ and $s$
is the normal coordinate near~$S$.

A 3-dimensional submanifold $L\subset S$ of a symplectic submanifold
$S\subset M$ is said to be \emph{special Lagrangian} if 
$\omega_0|_L=0$ and $\phi|_L=0$.
\end{dfn}
It is not difficult to see, using the pointwise model~\eqref{s2eq1}, that
every special Lagrangian submanifold of~$S$ is oriented by a
nowhere-vanishing 3-form $*_6\Upsilon_0=(\bfn_S\lrcorner*\phi)|_L$.

Definition~\ref{dfn3.2} extends the concept of special Lagrangian
submanifolds usually found in the literature to a more general
class of ambient manifolds, similar to~\cite{salur-SL}.  To
clarify this generalization we note the following.
\begin{prop}\label{CY}
Let $S$ be a symplectic submanifold of an almost $\text{\emph{G}}_2$-manifold.
The almost complex structure~$I$ on~$S$ is integrable if and only if
$d_S*_6\Upsilon_0=0$. Furthermore, in this case the K\"ahler metric defined
by $\omega_0$ is Ricci-flat.
\end{prop}
\begin{pf}
As $d\phi=0$ on $M$, we obtain from~\eqref{localphi} that $d_S\Upsilon_0=0$.  Thus,
$d_S\Omega_0=0$ if and only if $d_S*_6\Upsilon_0=0$. The proposition now follows by the arguments
in~\cite[\S 2]{hitchin-SL}; we omit the details.
\end{pf}
When $I$ is integrable the nowhere-vanishing $(3,0)$-form is automatically
holomorphic. A K\"ahler metric is Ricci-flat if and only if its restricted
holonomy group is contained in $\SU(3)$. A complex 3-fold
$(S,\omega_0,\Omega_0)$ endowed with the Ricci-flat K\"ahler metric
and a holomorphic $(3,0)$-form satisfying~\eq{ric-flat} is sometimes
called a \emph{Calabi--Yau 3-fold}.  Thus, a symplectic submanifold of an almost $\G2$-manifold is a natural
generalisation of a Calabi--Yau 3-fold by weakening the condition on the complex structure.

Deformations of compact closed special Lagrangian submanifolds in Calabi--Yau
manifolds are unobstructed: there is a theorem of a similar type to
Theorem~\ref{mclean} and due again to McLean~\cite[Theorem 3.6]{McLean}.
Salur \cite{salur-SL} showed that the integrability of the  
complex structure on the Calabi--Yau manifold is unnecessary for the
deformation theory result to hold.  Thus, the deformation theory
remains valid for symplectic submanifolds $S$ of an almost $\G2$-manifold
$(M,\phi)$. 

 The next result provides motivation for our choice of
boundary conditions for coassociative submanifolds later.
\begin{thm}[\mbox{cf.~\cite{salur-SL}}]\label{mclean-SL}  
Let $(S,\omega_0)$ be a symplectic submanifold of an almost
$\text{\emph{G}}_2$-manifold and let $L\subset S$ be a compact special
Lagrangian submanifold without boundary. Then
\begin{itemize}
\item[\emph{(a)}]\label{isometry}
$\jmath_L:\bfv\mapsto(\bfv\lrcorner\omega_0)|_L$ defines a vector bundle
isometry between the normal bundle of~$L$ in $S$, $\nu_S(L)$, and~$T^*L$.
\item[\emph{(b)}] The normal vector fields at~$L$ defining infinitesimal
special Lagrangian deformations correspond, via $\jmath_L$,
to closed and coclosed 1-forms on~$L$. Conversely, every closed and coclosed
1-form on~$L$ corresponds via $\jmath_L$ to a normal vector field $\bfv$
at~$L$ which is tangent to a path of special Lagrangian deformations of~$L$.
\end{itemize}
\end{thm}
\begin{pf}
This is immediate from the proof of the main theorem in \cite{salur-SL}.
\end{pf}

\subsection{Boundary conditions}\label{s3subs1}

From now on we assume that a compact coassociative submanifold
$N\subset M$ has non-empty boundary $\partial N$.
To achieve the Fredholm property of our deformation problem for a
coassociative submanifold $N$ with boundary, we need to impose a condition
that the boundary is confined to move in a suitable submanifold of $M$.  We
call this submanifold a scaffold, borrowing the terminology
of~\cite[Definition 1]{Butscher}, where the deformations of minimal
Lagrangians with boundary are studied.

\begin{dfn}\label{s3dfn2}
We say that an orientable 6-dimensional submanifold $S$ of $M$ is a
\emph{scaffold}  
for $N$ if
\begin{itemize}
\item[(a)] $\p N\subset S$, $\bfn\in\nu_M(S)|_{\p N}$ and
\item[(b)] $S$ is a symplectic submanifold of~$(M,\phi)$.
\end{itemize}
\end{dfn}

The condition (a) implies that $\p N\subset S$ is special Lagrangian, by
Lemma~\ref{SLCA}. Then, by Theorem~\ref{mclean-SL}, (b) ensures that
special Lagrangian deformations of $\p N$ in $S$ are unobstructed with a
smooth moduli space. Thus it is suggestive to consider coassociative
deformations of $N$ which remain `orthogonal to~$S$', i.e.~$\bfn_S$ is
tangent to the deformation.

We shall always assume that
$\bfn=\bfn_S|_{\p N}\in C^{\infty}(TN|_{\partial N})$ is a unit
inward-pointing normal vector field on $N$ at $\partial N$. Respectively,
$s_{\bfn}=s|_{N}$ is a normal coordinate near the boundary of~$N$. That is,
a map
\begin{equation}\label{collarn}
(x,s_{\mathbf{n}})\mapsto \exp_N\big(s_{\mathbf{n}}\mathbf{n}(x)\big)
\end{equation}
is defined for all $(x,s_{\bfn})\in \p N\times [0,\epsilon_N)$ and gives a
diffeomorphism of $\p N \times[0,\epsilon_N)$ onto a collar neighbourhood
$C_{\p N}$ of $\p N$ in~$N$. In particular, $\bfn=\frac{\p}{\p s_{\bfn}}$.

Recall from Theorem~\ref{mclean} that infinitesimal deformations of $N$ are
given by closed self-dual 2-forms $\alpha$ on $N$.
On the collar neighbourhood $C_{\p N}=T_S\cap N$ of $\p N$ in $N$,
we can write a self-dual 2-form $\alpha$ as 
\begin{equation}\label{bdryform1}
\alpha|_{C_{\p N}}=\xi_{s}\w\d s+*_s\xi_s
\end{equation}
for a 1-parameter family of 1-forms $\xi_s$ on $\p N$, 
where $*_s$ denotes the Hodge star on the submanifold
$(\p N)\times\{s\}\subset C_{\p N}$  corresponding to a fixed value of~$s$.

For a self-dual 2-form $\alpha$ on $N$, we easily calculate,
using~\eqref{bdryform1} and restricting to  $C_{\p N}$,
that $d\alpha=0$ implies that
\begin{equation}\label{bdrycondn1}
d_{\p N}\xi_s+\frac{\p}{\p s}(*_s\xi_s)=0
\quad\text{and}\quad d_{\p N}*_s\xi_s=0.
\end{equation}
The infinitesimal special Lagrangian deformations correspond to closed and
coclosed 1-forms by Theorem~\ref{mclean-SL}. Therefore, $d\alpha=0$ leads
to infinitesimal special Lagrangian deformations of the boundary if and
only if
\begin{equation}\label{bdrycondn2}
\frac{\partial}{\partial s}(*_s\xi_s)|_{s=0}=0.
\end{equation}

Notice that 
\begin{equation}\label{bdryform2}
d_{\p N}(\bfn\lrcorner\alpha)=-d_{\p N}\xi_0.
\end{equation}
Hence, the equations 
\begin{equation}\label{bdrycondn3}
d\alpha=0\;\,\text{on $N$}\quad\text{and}\quad d_{\p N}(\bfn\lrcorner\alpha)=0\;\,\text{on $\p N$},
\end{equation}
by \eq{bdrycondn1}, are equivalent to $d\alpha=0$ with condition
\eq{bdrycondn2}.

We can describe the boundary condition corresponding to \eq{bdrycondn2}
for normal vector fields $\bfv=\jmath_N^{-1}(\alpha)$ at~$N$ directly as
follows. Recall that, in a neighbourhood $T_S$ of $S$,
$\varphi|_{T_S}=\omega_s\w\d s+\Upsilon_s$ with $\omega_0$ closed as $S$ is
a scaffold, hence $\frac{\partial\Upsilon_s}{\partial s}|_{s=0}=0$ as
$d\varphi=0$.  The equation \eq{bdrycondn2} that $\alpha$ satisfies is
equivalent to the condition
$\frac{\partial}{\partial s}(\bfv_s\lrcorner\Upsilon_s)|_{s=0}=0,$
whence
\begin{equation}\label{bdrycondn5}
\left.\frac{\partial \bfv_s}{\partial s}\right|_{s=0}\lrcorner\Upsilon_0=0\quad\text{on $\p N$.}
\end{equation}
where we have expressed $\bfv$ on the collar neighbourhood of $\p N$ as a 1-parameter family of vector fields on $\p N$.

\begin{remark}
For a general $2$-form $\tilde\alpha=\alpha_\tau+\alpha_\nu\we \d s$, 
the familiar Dirichlet and Neumann boundary conditions are given by,
respectively, $\alpha_\tau=0$ and $\alpha_\nu=0$. For a self-dual $\alpha$,
the two conditions are equivalent and force $\alpha$ and the corresponding
normal vector field $\jmath_N^{-1}(\alpha)$ to vanish at each point of 
$\p N$.\label{eight} However, if $d\alpha=0$ and $\alpha$ vanishes
on the boundary then $\alpha=0$ by \cite[Lemma 2]{Cap}.
This may be understood as an extension of \cite[Theorem IV.4.3]{HarLaw}, which
states that there is a locally unique coassociative submanifold containing any
real analytic 3-dimensional submanifold upon which $\varphi$ vanishes.
\end{remark}

For a coassociative submanifold without boundary, a self-dual 2-form is
closed if and only if it is harmonic, i.e. satisfies
$\Delta\alpha=d^*_+d\alpha=0$. When there is a non-empty boundary a
harmonic self-dual 2-form need not in general be closed. The following
lemma is a direct corollary of~\cite[Proposition 3.4.5]{Schwarz} proved by
integration by parts.

\begin{lemma}\label{equiv}
For $\alpha\in\Omega^2_+(N)$, \ $d\alpha=0$ if and only if
\begin{equation}
d^*_+d\alpha=0\;\,\text{on $N$}\quad\text{and}\quad
\bfn\lrcorner d\alpha=0\text{ on $\p N$,}\label{sys2}
\end{equation}
where $d^*_+=\frac12(d^*+*d^*):\Omega^3(N)\rightarrow\Omega^2_+(N)$
is the $L^2$-adjoint of $d|_{\Omega^2_+}$.
\end{lemma}

Our next proposition relates the infinitesimal coassociative deformations
of $N$ to solutions of a Fredholm linear boundary value problem.
\begin{prop}\label{s3prop0}
For $p>1$ and $k\ge 1$, the map
\begin{multline}\label{map}
\alpha\in L^p_{k+1}\Omega^2_+(N)\to \big(d^*_+d\alpha,\;
d_{\p N}(\bfn\lrcorner\alpha),\; d_{\p N}(\alpha|_{\p N})\big)\in\\
L^p_{k-1}\Omega^2_+(N)\oplus
L^p_{k-\frac{1}{p}}\big(d\Omega^1(\p N)\oplus d\Omega^2(\p N)\big),
\end{multline}
is surjective. The kernel of~\eqref{map} consists of smooth forms and has dimension $b^1(\p N)$.
In particular, the map~\eqref{map} is Fredholm.
\end{prop}
\begin{pf}
By~\cite[Theorem 3.4.10]{Schwarz}, for each
$\eta\in\Omega^2(N)$, the solution $\alpha$ of $\Delta\alpha=\eta$ on~$N$
exists and is uniquely determined by its tangential
$\alpha|_{\p N}=\xi_\tau\in\Omega^2(\p N)$ and normal
$\bfn\lrcorner\alpha=\xi_\nu\in\Omega^1(\p N)$
components at the boundary. Now if $\eta\in\Omega^2_+(N)$, 
$\xi_\tau=*_{\p N}\xi_\nu$ and $\alpha$ is a solution of the latter
boundary value problem, then so is $*\alpha$, so $\alpha$ is self-dual
by uniqueness. Thus $\alpha\in\Omega^2_+(N)$ is uniquely
determined in this case by $\Delta\alpha$ and $\alpha|_{\p N}$.
As the manifold $\p N$ is compact and without boundary, it now follows from
the Hodge theory on $\p N$ that the values of 
$d_{\p N}(\bfn\lrcorner\alpha)$ and $d_{\p N}(\alpha|_{\p N})$ can be
prescribed independently and the operator~\eqref{map} is surjective.
These values are both zero precisely when $\bfn\lrcorner\alpha$
is a harmonic 1-form on $\p N$, which gives the dimension $b^1(\p N)$ of
the kernel of~\eq{map}.

We see from \cite[Theorem 3.4.10]{Schwarz}, that solutions to
$\Delta\alpha=0$ with $\alpha=\psi$ on $\p N$ are smooth if $\psi$ is
smooth. A similar result holds if the first derivatives of $\alpha=\psi$ on $\p N$ are
smooth, so the kernel of~\eq{map} consists of smooth forms. 
\end{pf}
In light of the work in this section, we shall be interested in the
following subspace of self-dual 2-forms
\begin{equation}\label{hsdbc}
\Omega^2_+(N)_{\text{bc}}
=\{\alpha\in 
\Omega^2_+(N):\;
\bfn\lrcorner d\alpha=0\text{ and }d_{\p N}(\bfn\lrcorner\alpha)=0
\text{ on $\p N$}\}.
\end{equation}
Our next result follows immediately from Lemma \ref{equiv} and Proposition \ref{s3prop0}.
\begin{cor}\label{subspace}
The image of $L^p_{k+1}\Omega^2_+(N)_{\mathrm{bc}}$ under the map~\eq{map}
is a closed subspace
\begin{equation}\label{target}
V^p_{k-1}\subseteq L^p_{k-1}\Omega^2_+(N)\oplus
L^p_{k-\frac{1}{p}}\big(d\Omega^1(\p N)\oplus
d\Omega^2(\p N)\big)
\end{equation}
The kernel of~\eq{map} intersects $L^p_{k+1}\Omega^2_+(N)_{\mathrm{bc}}$
in the subspace of (smooth) closed forms in $\Omega^2_+(N)_{\text{\emph{bc}}}$,
\begin{equation}\label{s3eq0}
(\mathcal{H}^2_+)_{\text{\emph{bc}}}=
 \big\{\alpha\in \Omega^2_+(N)\,:\,d\alpha=0\;\,\text{on $N$}
\text{ and }d_{\p N}(\bfn\lrcorner\alpha)=0\;\,\text{on $\p N$}\big\},
\end{equation}
and $\dim(\mathcal{H}^2_+)_{\text{\emph{bc}}}\le b^1(\p N)$.
\end{cor}
\noindent An example when strict inequality
$\dim(\mathcal{H}^2_+)_{\text{bc}}<b^1(\p N)$ occurs
is given in~\S\ref{examples}.

\section{Coassociative Local Deformations}\label{s4}

In this section, like in \S\ref{s3}, $(M,\phi)$ is an almost $\G2$-manifold
and $N\subset M$ is a compact coassociative submanifold with boundary in a
scaffold $S$.

We shall now define a version of a \emph{deformation map} $G$ whose
linearization gives the linear problem set up in the previous section.  The
role of $G$ for the study of deformations of $N$ with boundary in~$S$ is
similar to the role of $F$ in~\eqref{mapF} for the closed coassociative
submanifolds. However, our deformation map modifies~\eqref{mapF} in two
ways. First, we use the exponential mapping $\widehat{\exp}$ of a metric
$\hat{g}$ defined in \S\ref{s3subs2} which in general is not the metric
$g(\phi)$ induced by the $\G2$-structure~$\phi$.
 Second, our non-linear differential operator
$G$ is of second order, with the derivative at zero given by~\eq{map}
restricted to $\Omega^2_+(N)_{\text{bc}}$, defined in~\eq{hsdbc}.
An application of the Implicit Function Theorem to~$G$ will show that
the space of local coassociative deformations of~$N$ is smooth and has
finite dimension equal to that of the space of closed forms $(\mathcal{H}^2_+)_{\mathrm{bc}}$ in $\Omega^2_+(N)_{\text{bc}}$.

\subsection{Adapted tubular neighbourhoods}\label{s3subs2}

We wish to parameterize nearby deformations of $N$ with boundary in the
scaffold~$S$ by normal vector fields (or self-dual 2-forms on~$N$) via an
exponential map.
In general, we cannot use, as in \S\ref{s2}, exponential deformations of $N$ given by 
$g(\phi)$, since the scaffold may
not be preserved under these deformations. Therefore, we shall define
on~$M$ a modified metric whose related exponential map does preserve the
scaffold; that is, the scaffold is totally geodesic with respect to the new
metric. A similar approach was previously used in~\cite{Butscher} for
minimal Lagrangian submanifolds with boundary.

We first describe the local structure of the almost $\G2$-manifold $M$
near~$S$, applying a tubular neighbourhood argument (cf.~Proposition~\ref{s3thm1}).
\begin{lemma}\label{s3cor1}
There exist $\epsilon_S>0$, an open neighbourhood $T_S$ of $S$ in $M$
and a diffeomorphism $\eta_S:S\times(-\epsilon_S,\epsilon_S)\rightarrow
T_S$, such that 
$\eta_S(x,0)=x$, and $d\eta_S|_{(x,0)}
:T_xS\times\R\rightarrow T_xM$  satisfies
$d\eta_S|_{(x,0)}(\mathbf{0},1)=\bfn_S(x)$
 for all $x\in S$, where $\bfn_S$ is the unit normal to $S$ in $M$ as
 in~\S\ref{6submfd}.
 Furthermore, $\eta_S$ can be chosen so that
$$\eta_S(x,s_{\mathbf{n}})=\exp_N\big(s_{\mathbf{n}}\mathbf{n}(x)\big)$$
for all $x\in \partial N$ and $s_{\mathbf{n}}\in [0,\epsilon_S)$.
\end{lemma}
\begin{proof}
Recall that $\nu_M(S)$ is trivialised by $\mathbf{n}_S$. 
Therefore, from the proof of the Tubular Neighbourhood Theorem in~\cite{Lang}, 
we obtain $\epsilon>0$, an open neighbourhood $T$ of $S$ in $M$ and a
diffeomorphism $\eta:S\times (-\epsilon,\epsilon)
\rightarrow T$ given by
$$\eta(x,s_{\mathbf{n}})=
\exp_M\big(s_{\mathbf{n}}\mathbf{n}_S(x)\big).$$

The map $\eta$ satisfies all but the last of the required conditions, as
$\exp_M(s_{\bfn}\bfn)$ need not agree with $\exp_N(s_{\bfn}\bfn)$, even if
$0<\eps_S\le\eps_N$ (with $\eps_N$ as in the beginning of \S\ref{s3}).
However, since $T_{(x,0)}\eta^{-1}(N)=T_{(x,0)}\p N\oplus\RE\bfn$, 
a standard inverse mapping argument shows that, by composing $\eta$ with a
diffeomorphism $\psi$ of $S\times (-\epsilon,\epsilon)$ such that
$d\psi|_{(x,0)}=\id$, and choosing a sufficiently small $\eps_S>0$,  
we obtain $\eta_S:S\times(-\eps_S,\eps_S)\to T_S$ having all the required
properties.
\end{proof}
We can now construct the new metric.
\begin{prop}[cf.~\mbox{ \cite[Proposition 6]{Butscher}}]\label{newmetric}
There is a metric $\hat{g}$ on $M$, which equals $g$ outside of $T_S$, such that $S$ is totally geodesic with respect
 to $\hat{g}$.
\end{prop}
\begin{proof}
First recall Lemma \ref{s3cor1} and define a metric $h$ on $S\times (-\epsilon_S,\epsilon_S)$ by
$$h=\eta_S^*(g|_S)+\d s_{\mathbf{n}}\otimes\d s_{\mathbf{n}}.$$
Let $\chi:M\rightarrow[0,1]$ be a smooth function such that $\chi=0$ outside $T_S$ and $\chi=1$ in some tubular 
 neighbourhood
 of $S$ contained in $T_S$.  We then define $\hat{g}$ on $M$ by
$$\hat{g}=\chi(\eta_S^{-1})^*(h)+(1-\chi)g.$$
As in the proof of \cite[Proposition 6]{Butscher}, we see that $S$ is totally geodesic with respect to $\hat{g}$.
\end{proof} 
We shall need a variant of the isomorphism~\eqref{s2prop2} for the normal
bundle of $N$ with respect to $\hat{g}$.
\begin{prop}\label{newnu}
Let $\hat{\nu}_M(N)$ denote the normal bundle of $N$ relative to the metric
$\hat{g}$ given by Proposition \ref{newmetric}.
The map $\jmath_N:\hat{\nu}_M(N)\rightarrow(\Lambda^2_+)_gT^*N$ given by $\jmath_N(\mathbf{v})=(\mathbf{v}\lrcorner\varphi)|_{TN}$ is
an isomorphism.  Moreover, 
$\hat{\nu}_M(N)|_{\partial N}\subseteq TS|_{\partial N}$.
\end{prop}

\begin{proof} 
Recall from \eq{s2prop2} that $\jmath_N:\nu_M(N)\rightarrow(\Lambda^2_+)_gT^*N$ is an isomorphism.  Since 
$TM|_N=TN\oplus\nu_M(N)=TN\oplus\hat{\nu}_M(N)$, $\nu_M(N)\cong\hat{\nu}_M(N)$.  Moreover, the fibres of $\hat{\nu}_M(N)$ are 
transverse to those of $TN$ and $\jmath_N$ maps $TM|_N$ to $(\Lambda^2_+)_gT^*N$ since $\varphi$ vanishes on $TN$.  
So $\jmath_N$ defines an isomorphism between the vector bundles $\hat{\nu}_M(N)$
and $(\Lambda^2_+)_gT^*N$.

The final claims follows because $\hat{g}(x)$ coincides with $g(x)$ at
each \mbox{$x\in\p N$.}
\end{proof}
As a consequence of Propositions \ref{s3thm1}, \ref{newmetric} and \ref{newnu}
we obtain a version of the Tubular Neighbourhood Theorem which is adapted to local
deformations of~$N$ with boundary in~$S$.  Denote the exponential map on $M$
with respect to $\hat{g}$ by $\widehat{\exp}_M$.
However, we emphasise that the self-dual forms on~$N$ are always taken with
respect to the metric $g=g(\phi)$.  For an open subset $\mathcal{W}$ of a vector bundle $W$ on $N$, we define a
subset of the smooth sections $\Gamma(W)$ on $W$ by
$$\Gamma(\mathcal{W})=\{w\in \Gamma(W)\,:\,w(N)\subset \mathcal{W}\}.$$
We also make similar definitions for subsets of Banach spaces of sections when the Banach spaces consist of continuous sections.  
\begin{prop}\label{s3prop}
There exist an open subset $\mathcal{V}_N$ of $\hat{\nu}_M(N)$, containing the
zero section, and a 7-dimensional submanifold $\mathcal{T}_N$ of $M$ with
boundary, containing $N$, such that
$\widehat{\exp}_M:\mathcal{V}_N\rightarrow\mathcal{T}_N$
is a diffeomorphism such that if
$\mathbf{v}\in\Gamma(\mathcal{V}_N)$,
 then $\widehat{\exp}_M(\mathbf{v}(x))\in S$ for all $x\in\partial N$.

Respectively, $\mathcal{U}_N=\jmath_N(\mathcal{V}_N)$ is an open neighbourhood 
of the zero section in $\Lambda^2_+T^*N$ and $\delta_N=\widehat{\exp}_M\circ\jmath_N^{-1}:
\mathcal{U}_N\rightarrow\mathcal{T}_N$ is a diffeomorphism such that,  if
$\alpha\in\Gamma(\mathcal{U}_N)$, 
 then $\delta_N(\alpha(x))\in S$ for all $x\in\partial N$, so
$N_\alpha:=\delta_N(\alpha(N))\subset \mathcal{T}_N$ is a compact
4-dimensional submanifold of $M$ with boundary
$\partial N_{\alpha}\subset S$. 
\end{prop}

\subsection{The deformation map}

\begin{dfn}\label{s3dfn}\label{s4dfn2}
Let $\jmath_N$, 
 $\widehat{\exp}$, $\mathcal{V}_N$ and $\mathcal{U}_N$ be
as defined in~\S\ref{s3subs2}. Denote
\begin{equation}\label{coassoc}
\hat{F}:\alpha\in\Gamma(\mathcal{U}_N)\to 
\widehat{\exp}_{\bfv}^*\big(\varphi|_{N_{\alpha}}\big)\in\Omega^3(N),
\end{equation}
where $\bfv=\jmath_N^{-1}(\alpha)\in\Gamma(\mathcal{V}_N)$. 
We shall 
call the second order non-linear
differential operator
\begin{equation}\label{Gmap}
G=d_+^*\circ\hat{F}:\Gamma(\mathcal{U}_N)\to\Omega^2_+(N)
\end{equation}
the {\em deformation map} for~$N$.
\end{dfn}
Notice that the argument in~\cite[p.~731]{McLean} proving
Theorem~\ref{mclean}(a) does not depend on the choice of metric for the
exponential map, so we obtain the same result for $\hat{F}$ given
in~\eq{coassoc}.
\begin{lemma}\label{linearize} 
\begin{equation}\label{Flinear}
d\hat{F}|_0(\alpha)=d\alpha \quad\text{and}\quad
dG|_0(\alpha)=d_+^*d\alpha
\end{equation}
for all $\alpha\in\Gamma(\mathcal{U}_N)$.
\end{lemma}

Next we impose boundary conditions on the deformations $\p N$ in $S$, a
special Lagrangian submanifold.
By Theorem \ref{mclean-SL}, there exists an isomorphism $\jmath_{\p N}$ between $\hat{\nu}_S(\p N)$ 
and $T^*\p N$ (noting that we can use the normal bundle with respect to the metric $\hat{g}$).  From the neighbourhoods given in
Proposition \ref{s3prop}, one deduces that there exist open neighbourhoods 
$\mathcal{V}_{\p N}$ and $\mathcal{U}_{\p N}$ of the zero sections in $\hat{\nu}_S(\p N)$ and $T^*\p N$ respectively, 
with $\mathcal{U}_{\p N}=\jmath_{\p N}(\mathcal{V}_{\p N})$, and a tubular neighbourhood $T_{\p N}$ of $\p N$ in $S$ 
such that $\widehat{\exp}_S:\mathcal{V}_{\p N}\rightarrow \mathcal{T}_{\p N}$ is a diffeomorphism.  Define
\begin{equation}\label{FSL}
\hat{F}_{\p N}:
\beta\in\Gamma(\mathcal{U}_{\p N})
\to\\
\widehat{\exp}_{\bfv}^*(\bfn_S\lrcorner\varphi|_{\p N_\beta})
\in d\Omega^1(\p N),
\end{equation}
where $\bfv=\jmath_{\p N}^{-1}(\beta)$ and $\p
N_{\beta}=\widehat{\exp}_{\bf v}(\p N)$.  The kernel of
$\hat{F}_{\p N}$ characterises the Lagrangian
(but not necessarily special Lagrangian) local deformations of $\p N$ in~$S$.
The fact that $\hat{F}_{\p N}$ maps into the space of
forms claimed is a consequence of arguments in~\cite[\S 3]{McLean}
since the form $\bfn_S\lrcorner\varphi$ is exact near $\p N$ as it is closed
and vanishes on $\p N$. Define
\begin{equation}\label{gamma-bc}
\Gamma(\mathcal{U}_N)_{\mathrm{bc}}=\{\alpha\in\Gamma(\mathcal{U}_N)
\,:\,\hat{F}_{\p N}(\bfn\lrcorner\alpha)=0\text{ and }
\bfn\lrcorner \hat{F}(\alpha)=0\text{ on }\p N\}
\end{equation}
It follows from the deformation theory for $\p N$ in~$S$ and the work in
\S\ref{s3subs1} that, by taking completion in the appropriate Sobolev norm,
$L^p_{k+1}(\mathcal{U}_N)_{\mathrm{bc}}$ becomes a Banach submanifold of
$L^p_{k+1}\Omega^2_+(N)$ and its tangent space at $\alpha=0$ is
$L^p_{k+1}\Omega^2_+(N)_{\text{bc}}$ defined in~\eq{hsdbc}.

The next result shows that we can define coassociative local deformations with boundary in $S$ using a second order differential 
operator.
\begin{prop}\label{extmap}
For $\alpha\in\Gamma(\mathcal{U}_N)_{\text{\emph{bc}}}$, $G(\alpha)=0$ if and only
if $N_{\alpha}$ is coassociative.
For any coassociative deformation
$\alpha\in\Gamma(\mathcal{U}_N)_{\text{\emph{bc}}}$ defined by
$G(\alpha)=0$,  the local deformation of $\p N_{\bfn\lrcorner\alpha}\subset S$
is special Lagrangian.
\end{prop}
\begin{proof}
It is clear that $N_\alpha$ is coassociative precisely if
$\hat{F}(\alpha)=0$. Therefore, we suppose that $G(\alpha)=0$ and
show that then $\hat{F}(\alpha)=0$.

We know that the 3-form $\hat{F}(\alpha)$ is exact on~$N$ by the work in
\cite[\S 4]{McLean} since $\varphi$ is exact near $N$ as it is closed and
vanishes on $N$. As the last condition in~\eq{gamma-bc} asserts the
vanishing of the normal component $\bfn\lrcorner\hat{F}(\alpha)$ at the
boundary of $N$, the integration by parts argument applies to show that
$\hat{F}(\alpha)=0$.

For the last part of the Proposition, recall that the metric $\hat{g}$
constructed in the proof of Proposition~\ref{newmetric} is a product metric
near $S$, independent of the normal coordinate~$s$, and the exponential map
for $\hat{g}$ has the same $s$-invariant property. It can be checked that
$\widehat{\exp}_{\bfv|_{\p N}}^*(\varphi|_{\p N_{\bfn\lrcorner\alpha}})=\hat{F}(\alpha)|_{\p N}$
(in the notation of~\eq{FSL}). The latter vanishes since $\hat{F}(\alpha)=0$,
thus the Lagrangian deformation $\bfn\lrcorner\alpha$ is in fact special
Lagrangian.
\end{proof}

To apply the Banach space version of the Implicit Function
Theorem we note the following, by application of \cite[Theorem
2.2.15]{baier} or~\cite[\S 2.2]{joyce-salur}.
\begin{prop}\label{s4prop1}
The map $G$ given in~\eq{Gmap} 
extends to a smooth map of Sobolev
spaces $G:L^p_{k+1}(\mathcal{U}_N)_{\text{\emph{bc}}}\to L^p_{k-1}\Omega^2_+(N)$,
for any $p>4$, $k\geq 2$.
\end{prop}
\begin{remark} The conditions $p>4$ and $k\geq 2$ ensure that the map $G$ of Sobolev spaces in Proposition~\ref{s4prop1} is well-defined
since $L^p_{k+1}\hookrightarrow C^2$ in four dimensions, by the Sobolev Embedding Theorem.
\end{remark}

Recall the space $V^p_{k-1}$ defined in~\eq{target}. Let $\Pi$ denote the $L^2$-orthogonal projection
$L^p_{k-1}\Omega^2_+(N)\oplus L^p_{k-\frac{1}{p}}\big(d\Omega^1(\p N)\oplus
d\Omega^2(\p N)\big)\to V^p_{k-1}$ and define 
\begin{equation}\label{tildeGmap}
\widetilde{G}(\alpha)=\Pi\circ \big(G(\alpha),\hat{F}_{\p N}(\bfn\lrcorner\alpha),\bfn\lrcorner\hat{F}(\alpha)\big),
\end{equation}
where $\alpha\in L^p_{k+1}(\mathcal{U}_N)$ and $\hat{F}_{\p N}$ is given in~\eq{FSL}. 
As $L^p_k\Omega^2_+(N)_{\text{bc}}$ 
is the tangent space to $L^p_{k+1}(\mathcal{U}_N)_{\mathrm{bc}}$ at $\alpha=0$ and $G$ is
smooth, we find, by reducing the neighbourhood $\mathcal{U}_N$ if
necessary, that $\widetilde{G}(\alpha)=0$ if and only if
$\alpha\in L^p_{k+1}(\mathcal{U}_N)_{\mathrm{bc}}$ and $G(\alpha)=0$.

\begin{prop}\label{s4prop3}
Let $p>4$ and $k\ge 2$. If $\alpha\in L^p_{k+1}(\mathcal{U}_N)_{\text{\emph{bc}}}$
and $\widetilde{G}(\alpha)=0$ then $\alpha$ is smooth.
\end{prop}
\begin{proof}  We can apply to $\widetilde G$ the general elliptic regularity result
\cite[Theorem  6.8.2]{Morrey}, which implies that $C^2$ solutions to a
(nonlinear) second-order elliptic equation (with suitable boundary conditions) are smooth.
\end{proof}

Our first main theorem, Theorem \ref{thm1}, now follows from this technical result.
\begin{thm}\label{s4thm1}
Let $\phi\in\Omega^3_+(M)$ be a closed positive 3-form on a 7-manifold $M$.
 Recall the map $G$ given in~\eq{Gmap} and the space 
$\Gamma(\mathcal{U}_N)_{\text{\emph{bc}}}$ defined in~\eq{gamma-bc}.
 Let $N$
be a compact coassociative submanifold of $(M,\varphi)$ with non-empty
boundary $\p N\subset S$, where $S$ is a scaffold for~$N$. For $p>4$, $k\ge 2$, 
an $L^p_{k+1}$-neighbourhood of zero in the space
$$
\mathcal{M}(N,S)=
\{\alpha\in\Gamma(\mathcal{U}_N)_{\text{\emph{bc}}}\;:\; G(\alpha)=0\}
$$
of coassociative local deformations of~$N$ is a smooth manifold
 parameterized by the finite-dimensional vector space
$(\mathcal{H}^2_+)_{\text{\emph{bc}}}$ given in Corollary~\ref{subspace}.
\end{thm}
\begin{proof}
For $p>4$ and $k\ge 2$, let
$W=L^p_{k+1}(\mathcal{U}_N)_{\mathrm{bc}}$ and 
$X=L^p_{k+1}\Omega^2_+(N)_{\mathrm{bc}}$.
We see that $X$ is a Banach space and an open neighbourhood of zero in $X$
parameterizes an open neighbourhood of zero in the Banach submanifold
$W\subset L^p_{k+1}\Omega^2_+(N)$. Further, $\widetilde{G}$, given in~\eq{tildeGmap}, 
 satisfies $\widetilde{G}(W)\subset V=V^p_{k-1}$, $\widetilde{G}(0)=0$  
 and, by construction, $d\widetilde{G}|_0:X\to V$ is the
surjective linear operator~\eq{map}. 

Therefore, we can apply the Implicit Function Theorem to deduce that, as
$\widetilde{G}$ is a smooth map,  
the kernel of $\widetilde{G}$ in $W$ near zero is a manifold smoothly
parameterized by a neighbourhood of zero in the kernel
$(\mathcal{H}^2_+)_{\text{bc}}$ of $d\widetilde{G}|_0$ in $X$. By Proposition
\ref{s4prop3}, elements of the kernel of $\widetilde{G}$ near zero are smooth,
so by further reducing, if necessary, the neighbourhood $\mathcal{U}_N$ of the
zero section we obtain, noting also the comments after Proposition~\ref{s4prop1},
that $\widetilde{G}^{-1}(0)$ in $W$ near zero is exactly $\mathcal{M}(N,S)$.
\end{proof}

\subsection{Varying the {\boldmath $\textbf{G}_2$}-structure}

In this subsection we prove our second main result, Theorem~\ref{thm2}, which shows
that coassociative submanifolds with boundary are `stable' under small
perturbations of the $\G2$-structure on the ambient 7-manifold.
\begin{thm}\label{extn-ca-b}
Let $\phi(t)\in\Omega^3_+(M)$, $t\in\RE$, be a smooth path of
closed positive 3-forms on a 7-manifold $M$. Let $N$ be a compact coassociative submanifold of 
$(M,\varphi(0))$ with non-empty boundary $\p N\subset S$, where $S$ is a scaffold
for~$N$. Suppose 
that $\phi(t)|_N$ is exact on $N$ and $\bfn\lrcorner(\phi(t)|_{N})$ 
is exact on $\p N$ for each~$t$.

There is an $\eps>0$ and, for each $|t|<\eps$, a normal vector field
$\bfv(t)\in\Gamma(\hat\nu_M(N))$ 
smoothly depending on~$t$ and such that
$$\bfv(0)=0,\quad \p N(t)\subset S\quad\text{and}\quad \text{$\phi(t)$ vanishes on $N(t)=\widehat{\exp}_{\bfv(t)}(N)$.}$$
Here the normal bundle $\hat\nu_M(N)$ and the exponential map are taken
with respect to the metric $\hat{g}$ given by Proposition~\ref{newmetric} applied to $g(\phi(0))$.
\end{thm}

\begin{remarks} Observe that $S$ need not be a scaffold for $N(t)$ relative to the
$\G2$-structure $\phi(t)$ when $t\neq 0$.  Furthermore, the normal vectors $\bfv(t)$ 
can be chosen to satisfy the boundary condition~\eq{bdrycondn5}.
\end{remarks} 
\begin{pf}
Let $\jmath_N:\hat{\nu}_M(N)\rightarrow(\Lambda^2_+)_{g(0)}T^*N$ be the isomorphism given by Proposition \ref{newnu} and let $p>4$ and
$k\geq 2$. 
Let $W_0=L^p_{k+1}(\mathcal{U}_N)_0$ denote a Banach submanifold of 
$L^p_{k+1}(\mathcal{U}_N)_{\mathrm{bc}}$ such that the tangent space of 
 $W_0$ 
at the zero form is the $L^2$-orthogonal complement $X_0$ 
to $(\mathcal{H}^2_+)_{\mathrm{bc}}$ in $L^p_{k+1}\Omega^2_+(N)_{\text{bc}}$,
 in the metric $g_0=g(\phi(0))$.  Let $\Gamma(\mathcal{U}_N)_0$ be the subset of smooth sections in $W_0$.  
We use a `parametric' version of the deformation
map~\eqref{coassoc} which we still denote by $\hat F$,
$$
\hat{F}:(t,\alpha)\in\RE\times\Gamma(\mathcal{U}_N)
\rightarrow 
\widehat{\exp}_{\bfv}^*\big(\varphi(t)|_{N_{\alpha}}\big)\in \Omega^3(N),
$$
where $\bfv=\jmath_N^{-1}(\alpha)$. 
Further, define `parametric' versions of the maps 
$G=d^*_+\circ\hat{F}$ (where $d^*_+$ is calculated using $g_0$) and $\hat{F}_{\p N}$, given in~\eq{Gmap} and~\eq{FSL}, 
by replacing $\phi$ with $\phi(t)$ throughout, so these
maps now take an additional argument $t$. 
 Notice that, if
$$\widehat{G}(t,\alpha)=\big(G(t,\alpha),\hat{F}_{\p N}(t,\bfn\lrcorner\alpha),\bfn\lrcorner\hat{F}(t,\alpha)\big),$$
then its partial derivative $d_2\widehat{G}|_{(0,0)}$ acts on $(t,\alpha)\in\R\oplus L^p_{k+1}\Omega^2_+(N)$ as $(t,\alpha)
\to L(\alpha)$, where $L$ is the surjective linear operator~\eq{map}.  
Thus, the image of $d_2\widehat{G}|_{(0,0)}$ restricted to $\R\times X_0$ is 
the Banach space $V_0=V^p_{k-1}$  given in~\eq{target}.
Let $\Pi$ be the $L^2$-orthogonal projection to $V_0$ defined using the metric $g_0$ and let $\widetilde{G}=\Pi\circ\widehat{G}$.

 Since $\phi(t)|_N$ and its normal part on $\p N$ are exact for
each $t$, a simple adaptation of the argument in Proposition \ref{extmap} shows
that, for $\alpha\in\Gamma(\mathcal{U}_N)_0$, $\widetilde{G}(t,\alpha)=0$ if and only if $N_\alpha$ is
coassociative relative to $\phi(t)$.  Moreover, $\widetilde{G}(\R\times W_0)\subset V_0$, $\widetilde{G}(0,0)=0$  
and $d_2\widetilde{G}|_{(0,0)}:\R\oplus X_0\rightarrow V_0$ is an isomorphism.

By the Implicit Function Theorem for $\widetilde{G}(t,\alpha)$ and an analogous regularity result to Proposition \ref{s4prop3} 
(valid for $t$ small as $\widetilde{G}(0,\alpha)$ is elliptic and ellipticity is an open condition), 
there is a smooth map $h$ defined on a neighbourhood $E_0$ of zero in
$\RE\times(\mathcal{H}^2_+)_{\text{bc}}$, taking values in $X_0$, 
such
that $h(0,0)=0$ and
$$
\widetilde{G}\big(t,\alpha_0+h(t,\alpha_0)\big)=0,\qquad (t,\alpha_0)\in E_0,
$$
are all the zeros of $\widetilde{G}$ near $(0,0)$. The required $\bfv(t)$, for small
$|t|$, may be taken to be $\jmath_N^{-1}\big(h(t,0)\big)$.
\end{pf}

\subsection{Examples}\label{examples}

We now give some simple examples for this deformation theory.

\begin{ex}[$\G2$-manifolds with symplectic boundary]
Suppose $(M,\phi)$ is an (almost) $\G2$-manifold with boundary $\p M$.  Then
$\p M$ has trivial normal bundle and it receives an induced
$\SU(3)$-structure from $M$. If $\mathbf{n}_{\p M}\lrcorner\phi|_{\p M}$ is a closed form on $\p M$,
where $\mathbf{n}_{\p M}$ is the unit normal vector field at $\p M$, then $\p M$
is a scaffold for coassociative submanifolds of $N\subset M$ with boundary
in $\p M$ and with $\mathbf{n}_{\p M}$ tangent to~$N$. Our deformation
theory results apply to this situation.
\end{ex}

\begin{ex}[$\G2$-manifolds with nearly K\"ahler boundary]
Perhaps the most obvious $\G2$-manifold with boundary to study is the unit ball $B$ in $\R^7$.  The boundary of $B$ 
is the nearly K\"ahler 6-sphere $\mathcal{S}^6$.  Suppose $N$ is a compact coassociative submanifold of $B$ with 
boundary in $\mathcal{S}^6$.  Then $\p N$ is a \emph{Lagrangian} (also called \emph{totally real}) 
submanifold of $\mathcal{S}^6$: the non-degenerate, but not closed, 2-form on $\mathcal{S}^6$ vanishes on $\p N$.  
Current work in progress of the second author shows that, for any nearly K\"ahler 6-manifold, 
the deformation theory of 
a Lagrangian submanifold is expected to be \emph{obstructed} up to rigid motion.  Therefore, one could not hope, in general, for a smooth
 moduli space of deformations of $N$ in the unit ball $B$ in $\R^7$.  This negative result extends to any $\G2$-manifold
with nearly K\"ahler boundary, or any coassociative 4-fold with boundary in a nearly K\"ahler `scaffold'.  This gives 
another motivation for our definition of a scaffold.
\end{ex}

\begin{ex}[Product $\G2$-manifolds 1]
A K\"ahler complex 3-fold $(S,\omega)$ is called \emph{almost Calabi--Yau} if
it admits a nowhere vanishing holomorphic $(3,0)$-form $\Omega$. 
Then $M=S\times\mathcal{S}^1$ is an almost $\G2$-manifold with
$\G2$-structure $\omega\we \d\theta+\mathrm{Re}\,\Omega$, where $\theta$
is a coordinate on~$\mathcal{S}^1$. Let $N=L\times\mathcal{S}^1\subset M$ be a
compact coassociative 4-fold. Then $L$ is special Lagrangian in $S$.  
We can think of $N$ as an embedding of a manifold $L\times[0,1]$ whose two
boundary components, $L\times\{0\}$ and $L\times\{1\}$, are mapped to
$L$ in $S$. It is not difficult to see that $S\times\mathrm{pt}$ is a
scaffold for $N$. Theorem \ref{s4thm1} gives us that $N$ has a smooth
moduli space of coassociative deformations $\mathcal{M}(N,S)$ with
dimension $\le 2b^1(L)$.

Let $\alpha\in\Omega^2_+(N)$.  Then $\alpha=\xi_\theta\w\d\theta+*_L\xi_\theta$,
for some path of 1-forms $\xi_\theta$ on $L$.  
It follows from \cite[Theorem
3.4.10]{Schwarz} that a harmonic self-dual 2-form on $N$ is uniquely
determined by its values $\xi_0$, $\xi_1$ on the boundary.
The subspace of harmonic $\alpha\in\Omega^2_+(N)$ such that $\xi_0$ and $\xi_1$
are harmonic on~$L$ has dimension $2b^1(L)$ and corresponds
precisely to the paths $\xi_\theta=(1-\theta)\xi_0+\theta\xi_1$.
On the other hand, $\alpha\in (\HH^2_+)_{\text{bc}}$ if and only if $\alpha$
is harmonic and $\frac{\partial\xi_{\theta}}{\partial\theta}=0$, 
so $\xi_0=\xi_1$. Thus
$\dim(\HH^2_+)_{\text{bc}}=b^1(L)<b^1\big((L\times\{0\})\sqcup 
(L\times\{1\})\big)$ in this example.

This can also be seen geometrically.  If the deformations of the aforementioned 
two boundary components coincide in~$S\times\mathrm{pt}$ then, by taking a
product with $S^1$, we obtain a coassociative deformation of
$N=L\times\mathcal{S}^1$ defining a point in $\mathcal{M}(N,S)$. On the other
hand, if a coassociative deformation $\tilde{N}$ of $N$ is such that the
deformations $\tilde{L}_0$ and $\tilde{L}_1$ of $L\times\{0\}$ and $L\times\{1\}$ are special
Lagrangian but {\em distinct} then  $\tilde{N}$ and $\tilde{L}_0\times S^1$ are
two distinct coassociative 4-folds intersecting in a real analytic 3-fold on which $\phi$ vanishes, which
contradicts \cite[Theorem IV.4.3]{HarLaw}. Therefore, $\mathcal{M}(N,S)$ is
identified with special Lagrangian deformations of $L$ in the almost
Calabi--Yau manifold $S$. It is well known that these deformations
have a smooth moduli space of dimension $b^1(L)$ \cite[Theorem 8.4.5]{Joyce2}.

Moreover, suppose we have a smooth path of closed positive 3-forms $\phi(t)$ on $M$ with $\phi(0)=\phi$.  
Suppose $\phi(t)|_N$ is exact and the normal 
part of $\phi(t)|_N$ on $\p N$ is exact.  Theorem \ref{extn-ca-b} says that $N$ extends to a smooth family $N(t)$ 
of compact 4-folds with boundary in $S$ such that $N(t)$ is coassociative in $(M,\phi(t))$.  

Now, we can write 
$$\phi(t)=\omega(t)\w\d\theta+\Upsilon(t)$$
with $\omega(0)=\omega$ and $\Upsilon(0)=\Upsilon$.  The conditions on $\phi(t)$ are equivalent to the 
exactness of $\omega(t)|_L$ and $\Upsilon(t)|_L$ on $L$, together with the fact that $\omega(t)$ and $\Upsilon(t)$
 define an (almost) Calabi--Yau structure on $S$.  
These are precisely the necessary and sufficient conditions, by \cite[Theorem 8.4.7]{Joyce2}, 
for $L$ to be extended to a smooth family $L(t)$ of 
compact 3-folds in $S$ such that $L(t)$ is special Lagrangian in $S$ with respect to $(\omega(t),\Upsilon(t))$.
This applies to embeddings of $N(t)=L(t)\times\mathcal{S}^1$ and to the more
general embeddings of $N(t)=L(t)\times [0,1]$ with images of $L(t)\times\{0\}$ and 
$L(t)\times\{1\}$ in~$S\times\mathrm{pt}$.
\end{ex}

Finally we relate our work to the theory presented in \cite{Butscher}.

\begin{ex}[Product $\G2$-manifolds 2] 
Suppose $M=S\times\mathcal{S}^1$ is as in the previous example and $N=L\times\mathcal{S}^1$ is a coassociative 
4-fold in $M$.  However, now suppose that the special Lagrangian 3-fold $L$ has boundary $\p L$ in a (real) 4-dimensional scaffold $W$ in
 the almost Calabi--Yau manifold $S$ in the sense of \cite{Butscher}.  Thus $N$ has boundary 
$\p N=\p L\times\mathcal{S}^1$ in $W\times\mathcal{S}^1$.  The 5-dimensional submanifold $W\times\mathcal{S}^1$ is
 obviously not a scaffold in the sense of Definition \ref{s3dfn2}, so our deformation theory does not apply.  
However the result of Butscher's work \cite{Butscher} is that the deformation theory of $L$ as a \emph{minimal Lagrangian}
 in $S$ with boundary in $W$ is unobstructed and the dimension of the moduli space is $b^1(L)$.

Motivated by this example the authors considered the possibility of a 5-dimensional `scaffold' and derived the 
conditions that it would have to satisfy if it were to be a generalisation of Butscher's scaffold.  Unfortunately, 
the resulting deformation problem turned out not to be elliptic and, moreover, that the deformation theory of
coassociative submanifolds in such a 5-dimensional `scaffold' would be obstructed in general.  We can see 
this problem as follows.

The deformation theory we need for a coassociative $N$ with boundary in $W\times\mathcal{S}^1$ 
is for special Lagrangians with boundary in $W$. These are deformations of
minimal Lagrangians with `fixed phase' which means the vanishing of
$\mathrm{Im}(e^{i\lambda}\Omega)$ for some $\lambda\in\RE$ fixed once and for
all. However, the deformation theory of special Lagrangians with boundary
in $W$ is obstructed.  Therefore, even in this simplest case of a
5-dimensional `scaffold' we do not get a smooth moduli space of
deformations.
\end{ex}

\begin{ack}
The authors would like to thank Damien Gayet for pointing out an error in an earlier version of this article.  
The second author thanks Daniel Fox for interesting conversations.
\end{ack}

\end{document}